\documentclass[12pt]{amsart}

\usepackage{amssymb,latexsym,amsmath,extarrows, mathrsfs, amsthm}
\usepackage{graphicx}

\usepackage[margin=1in, centering]{geometry}

\usepackage{color}
\usepackage{wasysym}

\newtheorem{theorem}{Theorem}[section]
\newtheorem{lemma}[theorem]{Lemma}

\newtheorem{remark}[theorem]{Remark}

\newtheorem{corollary}[theorem]{Corollary}

\newcommand{\R}{\mathbb R}

\newcommand{\ZZ}{\mathbb{Z}}

\newcommand{\tM}

\makeatletter
\@namedef{subjclassname@2020}{\textup{2020} Mathematics Subject Classification}
\makeatother

\allowdisplaybreaks

\title{A homogeneous method for summation and its application}
\author{Zhipeng Lu}
\address{Shenzhen MSU-BIT University, 1 International University Park Road, Dayun New Town, Longgang District, Shenzhen, Guangdong Province, P.R. China}

\email{zhipeng.lu@hotmail.com}

\date{}

\begin{document}
\begin{abstract}
	We introduce a homogeneous method to deal with summations with homogeneous factors. Then we use it to compute main terms in the asymptotics of distance energy of square lattices in circles, which relates to the conjecture of distinct distances by Erd\H{o}s. 
\end{abstract}
\maketitle

\section{A simple homogeneous method in summation with a log factor}
In this section, we mainly introduce a homogeneous method to specifically facilitate summations with a log factor as follows
\begin{lemma}\label{lem-split log}
	If the function $f(x,y)>0$ is homogeneous, i.e. $f(kx,ky)=f(x,y), \forall k\in\R$, and integrable in $x$, then 
	\[\sum_{n\leq N}f(n,N)\log n\sim cN\log N,\]
	where $c=\int_{0}^{1}f(x,1)dx$.
\end{lemma}
Here $f(x)\sim g(x)$ always means $\frac{f(x)}{g(x)}\rightarrow1$ as $x$ tends to infinity. Results in this form seems new in the author's view, but they might have been used by other authors. Though clear enough by itself, we prove it by double counting as follows 
\begin{proof}
To deal with the summation, we introduce a double counting method to split the log factor out as follows. First, we partition the interval $[1,N]$ into $[\frac{m-1}{K}N,\frac{m}{K}N)$ for $m=1,\dots,K$. On each sub-interval, since $f$ is projective and continuous, we can easily squeeze the partial sum as
\begin{align}\label{eq-squeeze 1}
&\xi_m\sum_{\frac{m-1}{K}N\leq n<\frac{m}{K}N}\log n
<\sum_{\frac{m-1}{K}N\leq n<\frac{m}{K}N}f(n,N)\log n<\eta_m\sum_{\frac{m-1}{K}N\leq n<\frac{m}{K}N}\log n,
\end{align}
where $\xi_m=\min_{\frac{m-1}{K}\leq \frac{n}{N}<\frac{m}{K}}\{f(n/N,1)\}$ and $\eta_m=\max_{\frac{m-1}{K}\leq \frac{n}{N}<\frac{m}{K}}\{f(n/N,1)\}$.
Then we can asymptotically approximate the partial sum of $\log n$ by integral as 
\[\sum_{\frac{m-1}{K}N\leq n<\frac{m}{K}N}\log n\sim\frac{N}{K}\int_{m-1}^m(\log x+\log\frac{N}{K})dx\sim\frac{N}{K}(\log N-\log K+\log m)\sim\frac{N\log N}{K},\]
if we set $\log K=o(\log N)$, i.e. $K=N^{o(1)}$. Thus by (\ref{eq-squeeze 1}), the sum may be abbreviated to
\[\frac{1}{N\log N}\sum_{n=1}^{N-1}f(n,N)\log n\sim\frac{1}{K}\sum_{m<K}\theta_m\sim\int_{0}^1f(x,1)dx, \]
for some $\xi_m\leq\theta_m\leq\eta_m$, if $f(x,1)$ is (Riemann) integrable. 	
\end{proof}
If $f(x,y)$ is not homogeneous, but with deviation, say, $f(kx,ky)=k^\alpha f(x,y)$ for some $\alpha\in\R$, then (\ref{eq-squeeze 1}) is just scaled by $N^{\alpha}$ and the result becomes
\begin{corollary}
	If $f(x,y)$ is homogeneous of degree $\alpha\in\R$, i.e. $f(kx,ky)=k^{\alpha}f(x,y)$, and integrable in $x$, then
	\[\sum_{n\leq N}f(n,N)\log n\sim cN^{1+\alpha}\log N,\]
	where $c=\int_{0}^{1}f(x,1)dx$.
\end{corollary}
For the most obvious example, let $f(x,y)=x^{1+\alpha}/y,\alpha>-2$. Then it just tells us that $\sum_{n\leq N}n^{1+\alpha}\log n\sim \frac{1}{2+\alpha}N^{2+\alpha}\log N$, which is seen from obvious approximation by integral.

Moreover, the double counting method allows us to handle summation with other factors than just the log factor, provided that the factor behaves as well as

\begin{corollary}
	Suppose that the function $f(x,y)$ is homogeneous of degree $\alpha\in\R$ and integrable in $x$, and that $g(x)$ has the property that $g(N)\rightarrow\infty$ and $g(Nx)=g(N)+o(g(N))$ for $0<\delta(N)<x<1$ and $\delta(N)\rightarrow 0$ as $N\rightarrow+\infty$. Then
	\[\sum_{n\leq N}f(n,N)g(n)\sim cNg(N),\]
	where $c=\int_{0}^{1}f(x,1)dx$.
\end{corollary}
\begin{proof}
	Following the proof of Lemma \ref{lem-split log}, the summation of logarithms is substituted by that of $g(n)$. By the property of $g(x)$, we have for $K=1/\delta(N)$,	\[\sum_{\frac{m-1}{K}N\leq n\leq\frac{m}{K}N}g(n)\sim\int_{\frac{m-1}{K}N}^{\frac{m}{K}N}g(x)dx=\frac{N}{K}\int_{m-1}^{m}g(Nx/K)dx=\frac{N}{K}(g(N)+o(g(N)))\sim\frac{Ng(N)}{K}.\]
	Thus,
	\[\lim\limits_{K\rightarrow+\infty}\frac{1}{Ng(N)}\sum_{n\leq N}f(n,N)g(n)=\lim\limits_{K\rightarrow+\infty}\frac{1}{K}\sum_{m<K}\theta_m=w\int_{0}^1f(x,1)dx.\] 
\end{proof}
\begin{remark}
	If $g(x)=(\log x)^\beta$ for some $\beta>0$, then $g(Nm/K)=(\log N+\log(m/K))^\beta\sim(\log N)^{\beta}$ for $m\leq K$ and $\log K=o(\log N)$, so that similar to Lemma \ref{lem-split log} we have \[\sum_{n\leq N}f(n,N)(\log n)^\beta\sim cN(\log N)^\beta.\]
	Moreover, it might be generalized to arbitrary $g(x)$ of slow growth by appropriate double counting. Also, it would be interesting to derive the minor terms of the above summations.
\end{remark}
\section{Application to distance energy estimate}
Now we apply the above results to an explicit counting problem in discrete geometry or number theory. Actually, we found the homogeneous phenomena during studying the following problem. Let $P=[\sqrt{N}]\times[\sqrt{N}]$ be the square grid of size $N$, where $[x]$ denotes the set of integers ranging from $1$ to $\lfloor x\rfloor$. By studying the value distribution of $x^2+y^2$ on $P$, it can be estimated that $d(P):=|\{d(p,q)\mid p,q\in P\}|\sim c\frac{|P|}{\sqrt{\log|P|}}$ for some $c>0$. This becomes the initiating example for the Erd\H{o}s conjecture on distinct distances in the Euclidean plane $\R^2$, which says $d(P):=|\{d(p,q)\mid p,q\in P\}|\geq c\frac{|P|}{\sqrt{\log|P|}}$ for any finite set $P\subset\R^2$ and some absolute constant $c>0$. 

Guth and Katz \cite{GK} established the nearly optimal bound $d(P)\geq c\frac{|P|}{\log|P|}$. The essential object therein is what they call \textit{distance quadruples}, i.e. $Q(P)=:\{(p_1,q_1,p_2,p_2)\in P^4\mid d(p_1,q_1)=d(p_2,q_2)\}$. We call $|Q(P)|$ the \textbf{distance energy} of $P$, denoted by $E_{2}(P)$. Note that in the appendix of \cite{GK}, $E_2([\sqrt{N}]\times[\sqrt{N}])$ is estimated to be $\theta(N^3\log N)$ by counting line-line incidences in $\R^3$. In this section, we establish the asymptotics of $E_2(P)$ for $P$ being square lattices in circles resorting to our homogeneous method.

Denote $r(n):=|\{(a,b)\in\ZZ^2\mid a^2+b^2=n\}|$. On average, we have the following estimate:
\begin{lemma}[see (7.20) of Wilson \cite{Wilson}]\label{lem-sq sum est.}
	For any positive integer $k$ and $x\in\R_+$, we have
	\[\label{eq-k=2}
	\sum_{n\leq x}r^2(n)\sim 4x\log x+O(x).\]
\end{lemma}
More precise estimate on the distance energy on square grids takes us more effort to develop number theoretic methods. For convenience, we study lattice grids in circles, i.e. $P=\ZZ^2\cap B_{\sqrt{N}}(0,0)$, where $B_n(a,b)$ denotes the disk centered at $(a,b)$ with radius $n$. By results of the Gauss circle problem (see 1.4 of \cite{Karatsuba}), 
\begin{equation}\label{eq-circle problem}
|P|=\pi N+o(N^{1/3}).
\end{equation}

Denote by $r_{a,b}(n)=\{(x,y)\in P\mid (x-a)^2+(y-b)^2=n\}$ so that $r_{0,0}(n)=r(n)$ for $n\leq N$. Actually, if $\sqrt{a^2+b^2}\leq\sqrt{N}-\sqrt{n}$, then $r_{a,b}(n)=r(n)$. For $\sqrt{a^2+b^2}>\sqrt{N}-\sqrt{n}$, $\partial B_{\sqrt{n}}(a,b)$ is cut by $\partial B_{\sqrt{N}}(0,0)$. By easy calculation, the cut arc has angle $2\arccos\left(\frac{a^2+b^2+n-N}{2\sqrt{n(a^2+b^2)}}\right)$. Then by symmetry, one may expect that
\begin{equation}\label{eq-pts on arc}
r_{a,b}(n)\sim \tilde{r}_{a,b}(n):=\ \begin{cases}r(n), \text{ if }\sqrt{a^2+b^2}\leq\sqrt{N}-\sqrt{n}, n\leq N;\\
0, \text{ if }\sqrt{a^2+b^2}\leq\sqrt{n}-\sqrt{N}, n>N;\\
\frac{r(n)}{\pi}\arccos\left(\frac{a^2+b^2+n-N}{2\sqrt{n(a^2+b^2)}}\right), \text{ otherwise}.
\end{cases}
\end{equation}
Although the estimate by $\tilde{r}_{a,b}(n)$ may deviate from the true distribution, the summation $R(n):=\sum_{(a,b)\in P}r_{a,b}(n)$ counting all the pairs of points $(p,q)\in P^2$ with $d(p,q)=n$, turns out to be valid from the average symmetric point of view. We may use area counting to clarify this. Define $s_{a,b}(n)=|\{(x,y)\in\ZZ^2\mid (x-a)^2+(y-b)^2\leq n\}$ for any $(a,b)\in B_{\sqrt{N}}(0,0), 0\leq n\leq 4N$. Denote by $s(n)=s_{0,0}(n)$. Clearly by simple trigonometry \[s_{a,b}(n)-s_{a,b}(n-1)= \frac{s(n)-s(n-1)}{\pi}\arccos\left(\frac{a^2+b^2+n-N}{2\sqrt{n(a^2+b^2)}}\right)+O(1).\] 
Hence we have 
\begin{align}\label{eq-area counting}R(n)&=S(n)-S(n-1)=\sum_{a^2+b^2\leq N}(s_{a,b}(n)-s_{a,b}(n-1))\\
&=\sum_{\sqrt{a^2+b^2}\leq\sqrt{N}-\sqrt{n}}r(n)+\sum_{\sqrt{a^2+b^2}>\sqrt{N}-\sqrt{n}}\tilde{r}_{a,b}(n)+O(N).\notag
\end{align}
More explicitly, we show
\begin{lemma}\label{lem-Rn}
	Let $P$ be the integer points in the disk of radius $\sqrt{N}$ and $R(n)$ be the number of pairs of points from $P$ with distance $\sqrt{n}, n\leq 4N$ as above. Then 
	\[R(n)=\left(2\arccos\left(\frac{\sqrt{n/N}}{2}\right)-\sqrt{\frac{4Nn-n^2}{4N^2}}\right)Nr(n)+O(N).\] 
\end{lemma}
\begin{proof}
	By (\ref{eq-circle problem}), (\ref{eq-pts on arc}) and (\ref{eq-area counting}), we have for $n\leq N$,
	\begin{align*}R(n)&=  r(n)\sum_{\sqrt{a^2+b^2}\leq\sqrt{N}-\sqrt{n}}1+\frac{r(n)}{\pi}\sum_{\sqrt{N}-\sqrt{n}<\sqrt{a^2+b^2}\leq\sqrt{N}}\arccos\left(\frac{a^2+b^2+n-N}{2\sqrt{n(a^2+b^2)}}\right)+O(N)\\
	&=\pi r(n)(\sqrt{N}-\sqrt{n})^2+\frac{r(n)}{\pi}\iint_{(\sqrt{N}-\sqrt{n})^2\leq x^2+y^2\leq N}\arccos\left(\frac{x^2+y^2+n-N}{2\sqrt{n(x^2+y^2)}}\right)dxdy+O(N).
	\end{align*}
	Using the polar coordinates we may transform the double integral into
	\begin{align*}
	&2\pi\int_{\sqrt{N}-\sqrt{n}}^{\sqrt{N}}r\arccos\left(\frac{r^2+n-N}{2\sqrt{n}r}\right)dr\\
	=&\pi r^2\arccos\left(\frac{r^2+n-N}{2\sqrt{n}r}\right)\mid_{\sqrt{N}-\sqrt{n}}^{\sqrt{N}}+\pi\int_{\sqrt{N}-\sqrt{n}}^{\sqrt{N}}r^2\frac{\frac{1}{2\sqrt{n}}+\frac{N-n}{2\sqrt{n}r^2}}{\sqrt{1-\frac{(r^2+n-N)^2}{4nr^2}}}dr\\
	=&\pi N\arccos\left(\frac{\sqrt{n/N}}{2}\right)-\pi^2(\sqrt{N}-\sqrt{n})^2+\frac{\pi}{2}\int_{\sqrt{N}-\sqrt{n}}^{\sqrt{N}}\frac{r^2+N-n}{\sqrt{4nr^2-(r^2+n-N)^2}}d(r^2).
	\end{align*}
	Substituting by $s=\frac{r^2-n-N}{2\sqrt{Nn}}$ we get
	\begin{align*}
	\int_{-1}^{-\frac{\sqrt{n/N}}{2}}\frac{2\sqrt{Nn}s+2N}{\sqrt{1-s^2}}ds&=-2\sqrt{Nn}\sqrt{1-s^2}\mid_{-1}^{-\frac{\sqrt{n/N}}{2}}+2N\arcsin(s)\mid_{-1}^{-\frac{\sqrt{n/N}}{2}}\\
	&=-\sqrt{4Nn-n^2}+2N\left(\frac{\pi}{2}-\arcsin\left(\frac{\sqrt{n/N}}{2}\right)\right).
	\end{align*}
	Summing up everything provides us for $n\leq N$,
	\begin{align*}R(n)=&\pi r(n)(\sqrt{N}-\sqrt{n})^2+r(n)N\arccos\left(\frac{\sqrt{n/N}}{2}\right)-\pi r(n)(\sqrt{N}-\sqrt{n})^2\\
	&-\frac{r(n)}{2}\sqrt{4Nn-n^2}+\frac{\pi r(n)}{2}N-r(n)N\arcsin\left(\frac{\sqrt{n/N}}{2}\right)+O(N)\\
	=&r(n)\left(N\arccos\left(\frac{\sqrt{n/N}}{2}\right)-\sqrt{Nn-\frac{n^2}{4}}+\frac{\pi}{2}N-N\arcsin\left(\frac{\sqrt{n/N}}{2}\right)\right)+O(N)\\
	=&\left(2\arccos\left(\frac{\sqrt{n/N}}{2}\right)-\sqrt{\frac{4Nn-n^2}{4N^2}}\right)Nr(n)+O(N).
	\end{align*}
	When $N<n\leq 4N$, we have by (\ref{eq-pts on arc})
	\begin{align*}
	R(n)&=\frac{r(n)}{\pi}\sum_{\sqrt{n}-\sqrt{N}<\sqrt{a^2+b^2}\leq\sqrt{N}}\arccos\left(\frac{a^2+b^2+n-N}{2\sqrt{n(a^2+b^2)}}\right)+O(N)\\
	&=\frac{r(n)}{\pi}\iint_{(\sqrt{N}-\sqrt{n})^2\leq x^2+y^2\leq N}\arccos\left(\frac{x^2+y^2+n-N}{2\sqrt{n(x^2+y^2)}}\right)dxdy+O(N)\notag\\
	&=r(n)r^2\arccos\left(\frac{r^2+n-N}{2\sqrt{n}r}\right)\mid_{\sqrt{n}-\sqrt{N}}^{\sqrt{N}}+\frac{r(n)}{2}\int_{\sqrt{n}-\sqrt{N}}^{\sqrt{N}}r^2\frac{\frac{1}{2\sqrt{n}}+\frac{N-n}{2\sqrt{n}r^2}}{\sqrt{1-\frac{(r^2+n-N)^2}{4nr^2}}}dr+O(N)\\
	&=Nr(n)\arccos\left(\frac{\sqrt{n/N}}{2}\right)+\frac{r(n)}{2}\int_{\sqrt{n}-\sqrt{N}}^{\sqrt{N}}\frac{r^2+N-n}{\sqrt{4nr^2-(r^2+n-N)^2}}d(r^2)+O(N)\\
	&=Nr(n)\arccos\left(\frac{\sqrt{n/N}}{2}\right)+\frac{r(n)}{2}\left(-\sqrt{4Nn-n^2}+2N\left(\frac{\pi}{2}-\arcsin\left(\frac{\sqrt{n/N}}{2}\right)\right)\right)+O(N)\\
	&=\left(2\arccos\left(\frac{\sqrt{n/N}}{2}\right)-\sqrt{\frac{4Nn-n^2}{4N^2}}\right)Nr(n)+O(N),
	\end{align*}
	which adopts the same form as for $n\leq N$.
\end{proof}
As asymptotics of single $R(n)$, the above result seems too weak, but it provides us to the main term of the distance energy as follows
\begin{theorem}\label{thm-E2 of grids in circle}
	Let $P$ be the set of integral lattice points in a disk of radius $\sqrt{N}$, then 
	\[E_2(P)\sim(4\pi^2-8\pi+16)N^3\log N.\]
\end{theorem}
\begin{proof}
	Let $E(x)=\sum_{n\leq x}r^2(n)$. Then by Lemma \ref{lem-sq sum est.}, Lemma \ref{lem-Rn}, (\ref{eq-k=2}) and Abel's summation by parts, we get (noting that $\sum_{n\leq N}r(n)\sim\pi N$)
	\begin{align*}
	E_2(P)&=\sum_{n\leq 4N}R(n)^2= N^2\sum_{n\leq 4N}r^2(n)\left(2\arccos\left(\frac{\sqrt{n/N}}{2}\right)-\sqrt{\frac{4Nn-n^2}{4N^2}}\right)^2+O(N^3)\\
	&=4N^2\sum_{n=1}^{4N}r^2(n)\arccos^2\left(\sqrt{\frac{n}{4N}}\right)-2N^2\sum_{n=1}^{4N}r^2(n)\sqrt{\frac{4Nn-n^2}{N^2}}\arccos\left(\frac{\sqrt{n/N}}{2}\right)\notag\\
	&\quad +\frac{N^2}{4}\sum_{n=1}^{4N}r^2(n)\frac{4Nn-n^2}{N^2}+O(N^3)\notag\\
	&=\textcircled{1}-\textcircled{2}+\textcircled{3}+O(N^3).
	\end{align*}
    Using Abel summation and Lemma \ref{lem-sq sum est.}, we get
	\begin{align}
	\textcircled{1}=&4N^2\sum_{n=1}^{4N}(E(n)-E(n-1))\arccos^2\left(\sqrt{\frac{n}{4N}}\right)\notag\\
	=&4N^2\sum_{n=1}^{4N-1}E(n)\left(\arccos^2\left(\sqrt{\frac{n}{4N}}\right)-\arccos^2\left(\sqrt{\frac{n+1}{4N}}\right)\right)+O(N^2)\notag\\
	=&8N^2\sum_{n=1}^{4N-1}E(n)\arccos\left(\sqrt{\frac{n}{4N}}\right)\left(\arccos\left(\sqrt{\frac{n}{4N}}\right)-\arccos\left(\sqrt{\frac{n+1}{4N}}\right)\right)+O(N^2)\notag\\
	=&4N^2\sum_{n=1}^{4N-1}\frac{E(n)}{\sqrt{(4N-n)n}}\arccos\left(\sqrt{\frac{n}{4N}}\right)+O(N^{2}\log N)\notag\\
	=&16N^2\sum_{n=1}^{4N-1}\frac{\sqrt{n}\log n}{\sqrt{4N-n}}\arccos\left(\sqrt{\frac{n}{4N}}\right)+O(N^3).\notag
    \end{align}
    Now the above summation falls into the case of Lemma \ref{lem-split log}, which provides us \[\textcircled{1}\sim 16c_1N^3\log N,\ c_1=\int_{0}^1\sqrt{\frac{t}{1-t}}\arccos(\sqrt{t})dt=\frac{\pi^2-4}{8}.\]
    Similarly, 
    \begin{align*}
    \textcircled{2}&=2N^2\sum_{n=1}^{4N}(E(n)-E(n-1))\sqrt{\frac{4Nn-n^2}{N^2}}\arccos\left(\sqrt{\frac{n}{4N}}\right)\\
    &=2N^2\sum_{n=1}^{4N-1}E(n)\left(\sqrt{\frac{4Nn-n^2}{N^2}}\arccos\sqrt{\frac{n}{4N}}-\sqrt{\frac{4N(n+1)-(n+1)^2}{N^2}}\arccos\sqrt{\frac{n+1}{4N}}\right)\\
    &\quad+O(N^2)\\
    &=8N^2\sum_{n=1}^{4N-1}\left(\frac{1-2\sqrt{\frac{n}{4N}}}{\sqrt{1-\frac{n}{4N}}}\arccos\sqrt{\frac{n}{4N}}-\sqrt{\frac{n}{4N}}\right)\log n+O(N^3)\\
    &\sim 8(c_2-\frac{2}{3})N^3\log N,
    \end{align*}
    where $c_2=\int_{0}^{1}\frac{1-2\sqrt{t}}{\sqrt{1-t}}\arccos(\sqrt{t})dt=\pi-2-2c_1$.
    And also,
    \begin{align*}
    \textcircled{3}&=\frac{N^2}{4}\sum_{n=1}^{4N}(E(n)-E(n-1))\frac{4Nn-n^2}{N^2}\\
    &=\frac{N^2}{4}\sum_{n=1}^{4N-1}E(n)\left(\frac{4Nn-n^2}{N^2}-\frac{4N(n+1)-(n+1)^2}{N^2}\right)+O(N^2)\\
    &=\sum_{n=1}^{4N-1}(2n-4N+1)n\log n+O(N^3)\\
    &\sim\frac{32}{3}N^3\log N.
    \end{align*}
    Finally, altogether we get
    \begin{align*}
    E_2(P)&\sim\textcircled{1}-\textcircled{2}+\textcircled{3}\sim(16c_1-8c_2+16)N^3\log N\\
    &=(4\pi^2-8\pi+16)N^3\log N.
    \end{align*}
\end{proof}
\begin{remark}
Notice that the above summations (divided by $N\log N$) converge extremely slow. Say the last summation in $\textcircled{1}$, computing until $N=10^{11}$, the second decimal is not even stable.
\end{remark}

\end{document}